\documentclass[10pt]{amsart} %

\usepackage{float}

\usepackage{epsfig}
\usepackage{tikz}
\usepackage{array,amsmath, enumerate,  url, psfrag}
\usepackage{amssymb, amsaddr, fullpage}
\usepackage{graphicx,subfigure}
\usepackage{color}

\newtheorem{theorem}{Theorem}[section]
\newtheorem{lemma}[theorem]{Lemma}

\newtheorem{thm}{Theorem}[section]

\newtheorem{definition}[thm]{Definition}

\title{DP-$4$-colorability of two classes of planar graphs}

\author{Lily Chen$^{1}$ \and Runrun Liu$^{2}$ \and  Gexin Yu$^{2,3}$ \and Ren Zhao$^{1}$ \and Xiangqian Zhou$^{1,4}$}

\address{
$^{1}$\small School of Mathematical Sciences, Huaqiao University, China.\\
$^{2}$ School of Mathematics $\&$ Statistics, Central China Normal University, Wuhan 430079, China. \\
$^3$\small Department of Mathematics, The College of William and Mary, Williamsburg, VA, 23185, USA.\\
$^4$\small Department of Mathematics and Statistics, Wright State University, Dayton, Ohio, 45435
}

\thanks{The first author is supported by NSFC (No.11501223 and No.11701195), Science Foundation of the Fujian Province (No.2016J05009) and Scientific Research Funds of Huaqiao University (No.14BS311). The third author is supported in part by the NSA grant: H98230-16-1-0316 and NSFC grant (11728102). The fifth author is partially supported by the Minjiang Scholar Program hosted by Huaqiao University, Quanzhou, Fujian, China. }

\email{827261672@qq.com (Liu), gyu@wm.edu, xiangqian.zhou@wright.edu}

\begin{document}

\maketitle

\begin{abstract}
DP-coloring (also known as correspondence coloring) is a generalization of list coloring introduced recently by Dvo\v{r}\'ak and Postle (2017). In this paper, we prove that every planar graph $G$ without $4$-cycles adjacent to $k$-cycles  is DP-$4$-colorable for $k=5$ and $6$. As a consequence, we obtain two new classes of $4$-choosable planar graphs.  We use identification of verticec in the proof, and actually prove stronger statements that every pre-coloring of some short cycles can be extended to the whole graph.
\end{abstract}

\section{Introduction}

Graph coloring is one of the most important research topics in graph theory.  Let $[k]$ denote the set $\{i \in Z |1 \leq i  \leq k\}$. A {\em proper $k$-coloring} of a graph $G$ is a function $c:V(G) \to [k]$ such that $c(u) \neq c(v)$ for every edge $uv \in E(G)$. A  graph $G$ is called {\em $k$-colorable} if it has a proper $k$-coloring. The minimum value of $k$ such that $G$ is $k$-colorable is called the {\em chromatic number} of $G$, denoted by $\chi(G)$.

A well-known generalization of proper $k$-coloring is the concept of list coloring, introduced by Vizing \cite{VIZING}, and independently by Erd\H{o}s, Rubin, and Taylor \cite{ERT}. A {\em list assignment} $L$ assigns each vertex $v$ a set of available colors $L(v)$. A graph $G$ is $L$-colorable if $G$ has a proper coloring $c$ such that $c(v) \in L(v)$ for every $v \in V(G)$. A graph $G$ is called {\em $k$-choosable} if $G$ is $L$-colorable for every $L$ with $|L(v)| \geq k$ and the minimum integer $k$ for which $G$ is $k$-choosable is called the {\em list-chromatic number} of $G$, denoted by $\chi_l (G)$.

Since a proper $k$-coloring corresponds to an $L$-coloring with $L(v)=[k]$ for every $v \in V(G)$, we have that $\chi(G) \leq \chi_{l} (G)$. It is well-known that there exist graphs $G$ satisfying $\chi(G) < \chi_{l} (G)$. For example, the famous $4$-Color Theorem states that every planar graph is $4$-colorable; while Voigt \cite{VOIGT} found a planar graph that is not $4$-choosable. An interesting problem in graph coloring is to find sufficient conditions for a planar graph to be $4$-choosable. The next result is a good example in that direction; the cases $k=3$ and $k=6$ was due to Fijavz et al \cite{FJMS}, the case $k=4$ was done by Lam, Xu, and Liu \cite{LXL}, and the case $k=5$ was by Wang and Lih \cite{WL}.

\begin{theorem} \label{thm:no-3456-list-colorable }

Let $k$ be an integer with $3 \leq k \leq 6$. If $G$ is a planar graph without a cycle of length $k$, then $G$ is $4$-choosable.
\end{theorem}

Recently, Dvo\^{r}\'{a}k and Postle \cite{DP17} introduced $DP$-coloring (also known as correspondence coloring) as a generalization of list coloring.

\begin{definition}
Let $G$ be a simple graph with $n$ vertices and let $L$ be a list assignment for $V(G)$. For each edge $uv$ in $G$, let $M_{uv}$ be a matching between the sets $L(u)$ and $L(v)$ and let $\mathcal{M}_L = \{ M_{uv} : uv \in E(G)\}$, called the \emph{matching assignment}. Let $H_{L}$ be the graph that satisfies the following conditions
\begin{itemize}
\item each $u\in V(G)$ corresponds to a set of vertices $L(u)$ in $H_L$;
\item for all $u \in V(G)$, the set $L(u)$ forms a clique;
\item if $uv \in E(G)$, then the edges between $L(u)$ and $L(v)$ are those of $M_{uv}$; and
\item if $uv \notin E(G)$, then there are no edges between $L(u)$ and $L(v)$.
\end{itemize}
If $H_L$ contains an independent set of size $n$, then $G$ has a {\em $\mathcal{M}_L$-coloring}. The graph $G$ is {\em DP-$k$-colorable} if, for any matching assignment $\mathcal{M}_L$ in which $L(u)\supseteq[k]$ for each $u \in V(G)$, it has a $\mathcal{M}_L$-coloring. The minimum value of $k$ such that $G$ is DP-$k$-colorable is the {\em DP-chromatic number} of $G$, denoted by $\chi_{DP}(G)$.
\end{definition}

As in list coloring, we refer to the elements of $L(v)$ as colors and call the element $i \in L(v)$ chosen in the independent set of an $\mathcal{M}_L$-coloring as the color of $v$.  It is not hard to see that $DP$-coloring generalizes list coloring: one may simply choose the matching $M_{uv}$ to be the set $\{ (u,c_1)(v,c_2) | c_1 \in L(u),  c_2\in L(v), c_1=c_2\}$ for every edge $uv$ of $G$. So we know that $\chi_l (G) \leq \chi _{DP} (G)$. The inequality may be strict: for example, it is known that $\chi_l (C_{2k}) = 2$ while $\chi_{DP} (C_{2k}) = 3$, where $k\geq 2$ is an integer.

The notion of $DP$-coloring was used in Dvo\^{r}\'{a}k and Postle \cite{DP17} to prove that every planar graph without cycles of lengths from 4 to 8 is $DP$-$3$-colorable and therefore 3-choosable, solving a long-standing conjecture of Borodin~\cite{BORODIN}. In this paper, we study $DP$-$4$-colorability of planar graphs. More specifically,  we are interested in finding sufficient conditions for a planar graph to be $DP$-$4$-colorable.  The next result of Kim and Ozeki \cite{KO17} provides an important motivation for our research.

\begin{theorem} \label{thm:no-3456-DP-coloring}

For each $k \in \{3,4,5,6\}$, every planar graph without a cycle of length $k$ is $DP$-$4$-colorable.
\end{theorem}

We say two cycles in a graph $G$ are  {\em adjacent}  if they share at least one edge.  The next result of Kim and X. Yu \cite{KY17} strengthens Theorem~\ref{thm:no-3456-DP-coloring} in the cases of $k=3$ and $k=4$.

\begin{theorem}\label{thm:no-34-DP-coloring}
If a planar graph $G$ has no $4$-cycles adjacent to $3$-cycles, then $G$  is $DP$-$4$-colorable.
\end{theorem}

In this paper, we strengthen Theorem~\ref{thm:no-3456-DP-coloring} in the cases of $k\in \{4,5,6\}$. We prove the following two results.

\begin{theorem}\label{main1}
A planar graph without $4$-cycles adjacent to $5$-cycles is DP-$4$-colorable.
\end{theorem}

\begin{theorem}\label{main2}
A planar graph without $4$-cycles adjacent to $6$-cycles is DP-$4$-colorable.
\end{theorem}

As a corollary, for each $k\in \{5,6\}$, planar graph without $4$-cycles adjacent to $k$-cycles is $4$-choosable.  Note that this provides two new classes of $4$-choosable planar graphs.

The paper is organized as follows: we introduce some notions and prove some preliminary results in Section ~\ref{sec:prem}; the proofs for Theorem ~\ref{main1} and ~\ref{main2} are presented in Section ~\ref{proof1} and ~\ref{proof2}, respectively.

\section{Preliminaries} \label{sec:prem}

The following are some notions used in the paper. Suppose that $G$ is a planar graph embedded on the plane. Let $V$ be the set of vertices and let $F$ be the set of faces. A $k$-vertex ($k^+$-vertex, $k^-$-vertex, respectively) is a vertex of degree $k$ (at least $k$, at most $k$, respectively). The same notation will be applied to faces and cycles.  An $(\ell_1, \ell_2, \ldots, \ell_k)$-face is a $k$-face $[v_1v_2\ldots v_k]$ with $d(v_i)=\ell_i$ for $1 \leq i \leq k$. Let $C$ be a cycle of $G$. We use $int(C)$ (resp. $ext(C)$) to denote the sets of vertices located inside (resp. outside) the cycle $C$. The cycle $C$ is called {\em separating} if both $int(C)$ and $ext(C)$ are nonempty. Let $\mathcal{M}_L$ be a matching assignment for $G$. Then an edge $uv\in E(G)$ is {\em straight} if every $(u,c_1)(v,c_2)\in E(M_{uv})$ satisfies $c_1=c_2$. The next lemma follows immediately from (\cite{DP17}, Lemma 7).

\begin{lemma}\label{straight1}
Let $G$ be a graph with a matching assignment $\mathcal{M}_L$. Let $H$ be a subgraph of $G$ which is a tree. Then we may rename $L(u)$ for $u\in H$ to obtain a  matching assignment $\mathcal{M'}_L$ for $G$ such that all edges of $H$ are straight in $\mathcal{M'}_L$.
\end{lemma}

Throughout the rest of the paper, we will use $\mathcal{G}_1$ to denote the class of planar graphs with no $4$-cycles adjacent to $5$-cycles; and we will use $\mathcal{G}_2$ to denote the class of planar graphs with no $4$-cycles adjacent to $6$-cycles. We will consider a graph $G$ in $\mathcal{G}_1$ or $\mathcal{G}_2$ as a plane graph; in other words, we assume that $G$ is embeded on the plane. Just like for cycles, two faces are called {\em adjacent} if they share at least one common edge. We define a {\em $T_i$-subgraph (or $T_i$ for short)} of $G$  to be a subgraph of $G$ constructed by exactly $i$ adjacent $3$-faces.  The proofs for the next two lemmas are straightforward and thus omitted.

\begin{lemma} \label{property1}
Let $G$ be a graph in $\mathcal{G}_1$. Then the following are true:
\begin{itemize}
\item[(a)] A $3$-face cannot be adjacent to a $4$-face;
\item[(b)] $G$ contains no $T_i$-subgraphs for $i \geq 3$;
\item[(c)] If two $3$-faces $f_1$ and $f_2$ are adjacent, then every face adjacent to $f_1$ other than $f_2$ must be a $6^+$-face;
\item[(d)] Every $4^+$-vertex $v$ is incident to at most $(d_G(v)-2)$ $3$-faces.
\end{itemize}
\end{lemma}

\begin{lemma} \label{property2}
Let $G$ be a graph in $\mathcal{G}_2$. Then the following are true:
\begin{enumerate}[(a)]
\item  $G$ contains no $T_i$-subgraphs for $i \geq 5$; moreover, every $T_4$-subgraph of $G$ is isomorphic to the wheel graph $W_4$ with $4$ spokes.
\item For $i \in \{2, 3, 4\}$, every edge in a $T_i$-subgraph is either adjacent to two $3$-faces or   adjacent to a $3$-face and a $7^+$-face;
\item Each $5^+$-vertex $v$ is incident to at most $(d_G(v)-2)$ $3$-faces.
\end{enumerate}
\end{lemma}

A cycle $C$ of $G$ is called {\em bad} if there exists a $4^+$-vertex in $V(G) \setminus V(C)$ that has at least four neighbors on $C$. Otherwise $C$ is called {\em good}.  Clearly a $3$-cycle is always good. The next lemma follows easily from the definitions of $\mathcal{G}_1$, $\mathcal{G}_2$, and bad cycles.

\begin{lemma} \label{good-cycle}
If $G$ is in $\mathcal{G}_1$, then every $7^-$-cycle is good; if $G$ is in $\mathcal{G}_2$, then $G$ has a bad $8^-$-cycle $C$ if and only if $G$ contains a subgraph isomorphic to one of the  graphs shown in Figure ~\ref{badcycle} where $C$ is the outer $4$-cycle or $8$-cycle.
\end{lemma}

\begin{figure}[H]
\includegraphics[scale=0.45]{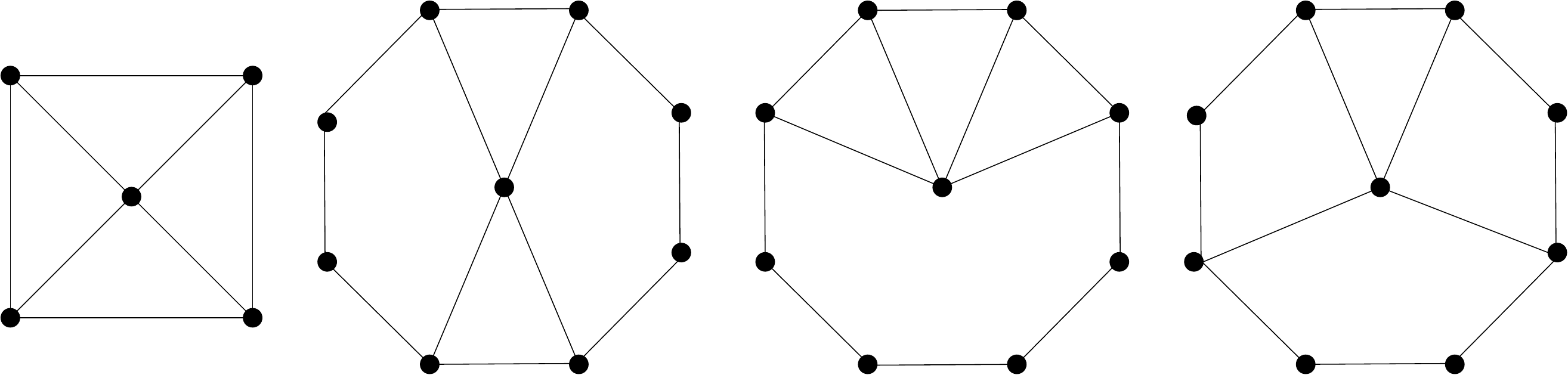}
\caption{ bad $4$-cycle or bad $8$-cycles.}
\label{badcycle}
\end{figure}

To prove our main results, we will choose a $7^-$-cycle (resp. a good $8^-$-cycle) $C_0$ if $G$ is in $\mathcal{G}_1$ (resp. $\mathcal{G}_2$) and assign a pre-DP-$4$-coloring $\phi_0$ on $V(C_0)$. Then we will show that the coloring $\phi_0$ can be extended to  a $DP$-$4$-coloring of $G$ by a discharge procedure on a minimal counterexample.  Assume $(G,C_0)$ is a minimal counterexample; that is, the coloring $\phi_0$ can not be extended to $G$; and $V(G)$ is as small as possible. We now prove some structural results for $(G, C_0)$.

\begin{lemma} \label{c0-nonseparating}
The cycle $C_0$ is not a separating cycle.
\end{lemma}
\begin{proof}
If $C_0$ is a separating cycle, then by the minimality of $(G,C_0)$, we can extend $\phi_0$ to $int(C_0)$ (resp. $ext(C_0)$).  So we get a $DP$-$4$-coloring of $G$ by combining them together, a contradiction.
\end{proof}

Since $C_0$ is non-separating, we may choose an embedding of $G$ such that $C_0$ is the boundary of the outer face $D$. A vertex $u$ of $G$ is called {\em internal} if $u \notin V(D)$; a face $f$ is called {\em internal} if $|V(f) \cap V(D)|=\emptyset$.

\begin{lemma}\label{minimum}
Every  internal vertex  has degree at least $4$.
\end{lemma}
\begin{proof}
Let $v$ be an internal $3^-$-vertex. By the minimality of $G$, $\phi_0$ can be extended to a $DP$-$4$-coloring $\phi$ of $G-v$. Then since $d_G(v) \leq 3$, we can extend $\phi$ to $G$ by selecting a color $\phi(v)$ for $v$ such that for each neighbor $u$ of $v$, $(u,\phi(u))(v,\phi(v))\notin E(M_{uv})$, a contradiction.
\end{proof}

\begin{lemma} \label{separating}
If $G \in \mathcal{G}_1$, then $G$ contains no separating $7^-$-cycles; and if $G \in \mathcal{G}_2$, then $G$ contains no separating good $8^-$-cycles.
\end{lemma}
\begin{proof} We will only present the proof for the case $G \in \mathcal{G}_2$; the proof for the case $G \in \mathcal{G}_1$ is similar.

Let $C$ be a separating good $8^-$-cycle in $G$. By the minimality of $G$, any precoloring of $C_0$ can be extended to $G-int(C)$. After that, $C$ is precolored, then again the coloring of $C$ can be extended to $int(C)$. Thus, we get a DP-4-coloring of $G$, a contradiction.
\end{proof}

\begin{lemma} \label{reduce}
Let $v$ be an internal $4$-vertex and  let $N(v)=\{v_i:1\le i\le4\}$ in a cyclic order in the embedding.  Suppose that $N(v)\cap D=\emptyset$.  Then

\begin{enumerate}[(a)]
\item if $G \in \mathcal{G}_1$, then at most one of $v_i$ and $v_{i+2}$ is a $4$-vertex for $i=1,2$.
\item Suppose that $G \in \mathcal{G}_2$. If $G$ contains no bad $4$-cycle $C$  with $\{v_1,v,v_3\} \subset V(C)$ and the graph obtained by identifying $v_1$ and $v_3$ in $G-\{v,v_2,v_4\}$ has no $4$-cycles adjacent to $6$-cycles, then at most one of $v_2$ and $v_4$ is a $4$-vertex.
\end{enumerate}
\end{lemma}

\begin{proof} Again we will only present the proof for the case $G \in \mathcal{G}_2$. The proof for the case  $G \in \mathcal{G}_1$ is similar.

By Lemma~\ref{straight1}, we can rename each of $L(v_1),L(v),L(v_3)$ to make $M_{vv_1}$ and $M_{vv_3}$ straight. Let $G'$ be the graph by identifying $v_1$ and $v_3$ of $G-\{v_2,v,v_4\}$ and let $\mathcal{M}'_L$ be the restriction of $\mathcal{M}_L$ to $E(G')$. Since $\{v_1,v_3\}\cap D=\emptyset$, the identification does not create an edge between vertices of $D$, and thus $\phi_0$ is also a $DP$-$4$-coloring of the subgraph $D$ of $G'$. Note that $G'$ contains no $4$-cycles adjacent to $6$-cycles by the assumption. Also $G'$ contains no loops or parallel edges, since $v_1,v,v_3$ are not on a bad $4$-cycle and $G$ has no separating $3$-cycles or separating good $4$-cycles. Therefore, $\mathcal{M}'_L$ is also a matching assignment on $G'$. Since $|V(G')|<|V(G)|$,  $\phi_0$ can be extended to  a $DP$-$4$-coloring $\phi$ of $G'$. For $x\in\{v_2,v,v_4\}$, let $L^*(x)=L(x)\setminus\cup_{ux\in E(G)}\{c'\in L(x):(u,c)(x,c')\in M_{ux}$ and $(u,c)\in \phi\}$. Then $|L^*(v_2)|=|L^*(v_4)|\ge1$, and $|L^*(v)|\ge3$. So we can extend $\phi$ to a $DP$-$4$-coloring of $G$ by coloring $v_1$ and $v_3$ with the color of the identified vertex and then color $v_2,v_4,v$ in order, a contradiction.
\end{proof}

\section{Proof of Theorem~\ref{main1}}\label{proof1}

In this section we prove the following result, which is stronger than Theorem~\ref{main1}. Recall that we use $\mathcal{G}_1$ to denote the class of planar graphs without $4$-cycles adjacent to $5$-cycles.

\begin{theorem}\label{strong1}
If $G \in \mathcal{G}_1$, then any pre-coloring of a $7^-$-cycle can be extended to a $DP$-$4$-coloring of $G$.
\end{theorem}

Let $(G,C_0)$ be a minimal counterexample to Theorem~\ref{strong1}.  That is, $G\in \mathcal{G}_1$, $C_0$ is a $7^-$-cycle in $G$ that is pre-colored with a $DP$-$4$-coloring $\phi_0$ that can not be extended to a $DP$-$4$-coloring of $G$, and $|V(G)|$ is as small as possible. Consider a planar embedding of $G$. By Lemma~\ref{c0-nonseparating}, we may assume that $C_0$ is the boundary of the outer face $D$ of $G$.

Now for each $x\in V\cup F\backslash \{D\}$, let $x$ have an initial charge of $\mu(x)=d(x)-4$, and $\mu(D)=d(D)+4$. By Euler's Formula, $\sum_{x\in V\cup F}\mu(x)=0$. Let $\mu^*(x)$ be the charge of $x\in V\cup F$ after the discharging procedure. To lead to a contradiction, we shall prove that $\mu^*(x)\ge 0$ for all $x\in V\cup F$ and $\mu^*(D)>0$.

We call a subgraph of $G$ a {\em{diamond}} if it contains two adjacent $3$-faces and the two common vertices of the $3$-faces are both $4$-vertices. The discharging rules:

\begin{enumerate}[(R1)]
\item Every internal $5^+$-vertex gives its initial charge evenly to its incident $3$-faces;
\item Every $5$-face $f \ne D$ gives $\frac{1}{3}$ to each adjacent internal $(4,4,4)$-face and $\frac{1}{6}$ to every other adjacent internal $3$-face;
\item Every $6^+$-face $f\ne D$ gives $\frac{1}{2} t $ to every adjacent internal $3$-face that is in a diamond and gives $\frac{1}{3} t $ to every adjacent internal $3$-face that is not in a diamond, where $t$ is the number of common edges of $f$ and the $3$-face;
\item The outer-face $D$ gets $\mu(v)$ from each incident vertex and gives $1$ to each non-internal $3$-face.
\item After the above rules, each face other than $D$ gives its surplus charge to $D$.
\end{enumerate}

Note that, by (R1) and ~\ref{property1} part (d), every internal $3$-face gets at least $\frac{1}{3}$ from every incident $5^+$-vertex $v$. Furthermore, if $v$ is not a $5$-vertex adjacent to three $3$-faces, then $f$ gets at least $\frac{1}{2}$ from $v$. Also note that, a vertex $v$ on the outer face may have degree $2$ or $3$; in these cases, $v$ will give a negative charge to the outer face $D$ by (R4), in other words, the vertex $v$ actually receives a positive charge from $D$.

We first check the final charge of vertices in $G$. Let $v$ be a vertex in $G$. If $v\in V(D)$, then by (R4) $\mu^*(v)\ge0$. If $v\notin V(D)$, then by Lemma ~\ref{minimum}, $d(v)\ge 4$.
If $d(v)=4$, then $v$ is not involved in the discharging procedure. So $\mu^*(v)=\mu(v)=d(v)-4=0$. If $d(v)\ge5$, then by (R1) $\mu^*(v)\ge0$.

Now we check the final charge of faces other than $D$ in $G$. Let $f$ be a face in $G$. By the rules, we only need to show that after (R1)-(R4), $f$ has a nonnegative charge, and we denote this charge by  $\mu'(f)$. If $d(f)=4$, then $f$ is not involved in the discharging procedure. So $\mu'(f)=\mu(f)=d(f)-4=0$. Suppose that $d(f)=5$ and $v_1$, $v_2$, $v_3$, $v_4$, and $v_5$ are the vertices on $f$ in cyclic order. If $f$ is adjacent to at most one internal $(4,4,4)$-face. then by (R2) $\mu'(f)\ge5-4-\frac{1}{3}-\frac{1}{6}\cdot 4=0$. By Lemma~\ref{reduce}, $f$ is adjacent to at most two internal $(4,4,4)$-faces. Assume that $f_1$ and $f_2$ are the internal $(4,4,4)$-faces that are adjacent to $f$. Then the edge shared by $f$ and $f_1$ is not adjacent to the edge shared by $f$ and $f_2$. By symmetry, assume that  $v_1, v_2 \in V(f_1)$ and $v_3,v_4 \in V(f_2)$. Now by Lemma~\ref{reduce}, the edge $v_2v_3$ is not adjacent to any internal $3$-face. Therefore, $\mu'(f) \geq  5-4 - \frac{1}{3} \cdot 2 -\frac{1}{6} \cdot 2 =0$.

Suppose that $d(f)\ge6$. If $f$ is adjacent to $m\ge0$ $3$-faces in diamonds, then $f$ is adjacent to at least $\lceil\frac{m}{2}\rceil$ $6^+$-faces. So
$f$ is adjacent to at most $(d(f)-m-\lceil \frac{m}{2}\rceil)$ $3$-faces not in diamonds. So $\mu'(f)\ge d(f)-4-\frac{1}{2}m-\frac{1}{3}(d(f)-m-\lceil\frac{m}{2}\rceil)\ge\frac{2}{3}d(f)-4\ge0$.

Suppose $d(f)=3$. If $f$ is not internal , then $\mu'(f)\ge-1+1=0$ by (R4). So we may assume that $f$ is an internal $3$-face. Let $f=[uvw]$ and let $f_1, f_2$ and $f_3$ be the faces sharing edges $uv$, $vw$, $uw$ with $f$, respectively. By Lemma~\ref{property1} part (b), at most one of $f_1,f_2$ and $f_3$ is a $3$-face.

\textbf{Case 1:} None of $f_1, f_2$ and $f_3$ is a $3$-face.

Then every face adjacent to $f$ is a $5^+$-face by Lemma ~\ref{property1} part (a). If $f$ is not adjacent to any $5$-faces, then by (R3) $f$ gets at least $\frac{1}{3}\cdot 3$ from adjacent faces. So $\mu'(f)\ge-1+1=0$. Therefore we may assume that $f$ is adjacent to a $5$-face, by symmetry assume that $f_1$ is a $5$-face.  If $f$ is a $(4,4,4)$-face, then by (R2) and (R3) $f$ gets at least $\frac{1}{3} \cdot 3$ from adjacent faces. So $\mu'(f)\ge-1+1=0$.  If $f$ is adjacent to at least two $5^+$-vertices, then $f$ gets at least $\frac{1}{3}\cdot2$ from the two $5^+$-vertices and $\frac{1}{6}\cdot3$ from adjacent faces. So $\mu'(f)\ge-1+\frac{2}{3}+\frac{1}{2}>0$. Therefore, we may assume that $f$ contains exactly one $5^+$-vertex. By symmetry, there are two cases: either $d(w) \geq 5$ or $d(v) \geq 5$.

First assume that $d(w) \geq 5$. Then $f$ gets $\frac{1}{6}$ from $f_1$ by (R2). If $d(w)\ge6$ or $d(w)=5$ and $w$ is incident to at most two $3$-faces, then $w$ gives at least $\frac{1}{2}$ to $f$ by (R1). Note that $f$ gets at least $\frac{1}{6}\cdot2$ from $f_2$ and $f_3$. So $\mu'(f)\ge-1+\frac{1}{6}\cdot3+\frac{1}{2}=0$. If $d(w)=5$ and $w$ is incident to at least three $3$-faces, then by Lemma ~\ref{property1} parts (b) and (c), $w$ is on exactly three $3$-faces and both $f_2$ and $f_3$ are $6^+$-faces. So $f$ gets at least $\frac{1}{3}\cdot3$ from $f_2$, $f_3$, and $w$;  and hence, $\mu'(f)\ge-1+\frac{1}{6}+1>0$.

Next we assume that $d(v) \geq 5$. Since $f_1$ is a $5$-face, $v$ can not be a $5$-vertex adjacent to exactly three $3$-faces. Therefore, $v$ is either a $6^+$-vertex or a $5$-vertex adjacent to at most two $3$-faces. So $f$ gets at least $\frac{1}{2}$ from $v$ by (R1). Note that $f$ gets at least $\frac{1}{6}\cdot3$ from adjacent faces. So $\mu'(f)\ge-1+\frac{1}{2}+\frac{1}{2}=0$.

\textbf{Case 2:} One of $f_1,f_2$ and $f_3$ is a $3$-face.

By symmetry assume that $f_1=[uvx]$ is a $3$-face. By Lemma~\ref{property1} part (c), the faces that share edges $vw$ and $uw$ with $f$ are $6^+$-faces. If both $u$ and $v$ are $4$-vertices, then $f$ is in a diamond. So $f$ gets $\frac{1}{2}\cdot2$ from adjacent $6^+$-faces. If at least one of $u$ and $v$ is a $5^+$-vertex, then $f$ gets at least $\frac{1}{3}$ from the $5^+$-vertex and $\frac{1}{3}\cdot2$ from adjacent $6^+$-faces. In both cases, $\mu'(f)\ge-1+1=0$.

\medskip
Finally, we check the final charge of $D$ and show that $\mu^*(D)>0$. Let $f_3$ be the number of non-internal $3$-faces and let $b$ be the charge that $D$ gets from other faces by (R5). Let $E(D, V(G)-D)$ be the set of edges between $D$ and $V(G)-D$ and let $s$ be its size. Recall that a $T_i$-subgraph of $G$ consists of $i$ adjacent $3$-faces. Since $G \in \mathcal{G}_1$, $G$ contains no $T_i$-subgraphs with $i \geq 3$. Therefore, every non-internal $3$-faces is either in a $T_1$-subgraph or a $T_2$-subgraph of $G$. We require the following lemma about the outer face $D$.

\begin{lemma} \label{lemma:D1} The following statements are true about $D$.
\begin{itemize}
\item[(a)] $D$ has no chord, and moreover, if $v_1,v_2\in V(D)$ and $v_1v_2\notin E(G)$, then $v_1$ and $v_2$ have no common neighbors in int(D).
\item[(b)] For $i \in \{1,2\}$, a $T_i$-subgraph of $G$ shares at most one edge with $D$.
\item[(c)] Let $f$ be a $5$-face that shares $k$ edges with $D$. If $k=1$, then $f$ gives at least $\frac{1}{2}$ to $D$; if $k=2$, then $f$ gives at least $\frac{2}{3}$ to $D$; if $k=3$, then $f$ gives $1$ to $D$.
\item[(d)] Let $f$ be a $6^+$-face that shares $k$ edges with $D$. Then $f$ gives at least $\frac{1}{3}(k+1)$ to $D$.
\end{itemize}
\end{lemma}
\begin{proof} (a). Suppose otherwise that $D$ has a chord. Since by Lemma~\ref{separating} $G$ has no separating $7^-$-cycle, $V(D)=V(G)$. So the coloring on $D$ is also a $DP$-$4$-coloring of $G$, a contradiction. Now assume that $v_1$ and $v_2$ are two non-adjacent vertices on $D$ that have a common neighbor $u \in int(D)$. Then by Lemma~\ref{minimum} $d(u)\ge 4$. Recall that $D$ is a $7^-$-face. If all neighbors of $u$ are on $D$, then it is routine to check that $G$ would contain a $4$-cycle adjacent to a $5$-cycle, a contradiction. Therefore, $u$ has a neighbor in int(D). It follows that $u$ is contained in a separating $7^-$-cycle, contrary to Lemma~\ref{separating}.

(b). This part follows easily from (a).

(c). If $k=1$, then $f$ is adjacent to at most two internal $3$-face and at most one of which is a $(4,4,4)$-face by Lemma~\ref{reduce}. So $f$ gives at least $1-\frac{1}{3}-\frac{1}{6}=\frac{1}{2}$ to $D$ by (R2) and (R5). If $k=2$, then $f$ is adjacent to at most one internal $3$-face, which gets at most $\frac{1}{3}$ from $f$. So $f$ gives at least $\frac{2}{3}$ to $D$ by (R2) and (R5). If $k=3$, then $f$ is adjacent to no internal $3$-face. So $f$ gives $1$ to $D$ by (R5).

(d). Note that by (R3) every $6^+$-face $f\ne D$ gives $\frac{1}{2} t $ to every adjacent internal $3$-face that is in a diamond and gives $\frac{1}{3} t $ to every adjacent internal $3$-face that is not in a diamond, where $t$ is the number of common edges of $f$ and the $3$-face. Suppose that $f$ is adjacent to an internal $3$-face $f_0$ that is in a diamond. Let $e$ be a common edge of $f$ and $f_0$. Then by Lemma ~\ref{property1} part (b), there exists an edge $e_0$ of $f$ incident to $e$ that cannot be on a $3$-face. So we can split the charge of $\frac{1}{2}$ between $e$ and $e_0$ such that $e$ carries $\frac{1}{3}$ and $e_0$ carries $\frac{1}{6}$. Furthermore, since $e_0$ may be incident to at most two edges of $f$ that are in a diamond, $e_0$ carries at most $\frac{1}{6}\cdot2=\frac{1}{3}$. Therefore, we may treat all the charge that $f$ sends out as at most $\frac{1}{3}$ through each edge of $f$.
Moreover, a common edge of $f$ and $D$ will not carry any charge; and if $e$ is an edge of $f$ that has exactly one vertex on $D$, then $e$ is incident to at most one edge of $f$ that belongs to an internal triangle within a diamond. So the edge $e$ carries at most a charge of $\frac{1}{6}$.  Now assume that $f$ shares $k$ edges with $D$. Since $D$ is chordless, these $k$ edges form a path. So the face $f$ sends out a charge of at most $\frac{1}{3} \cdot (d(f)-(k+2)) + 2 \cdot \frac{1}{6} = \frac{1}{3} \cdot (d(f)-(k+1))$. Therefore, after (R1)-(R4), the remaining charge of $f$ is at least $d(f)-4 - \frac{1}{3} (d(f)-(k+1)) \geq \frac{1}{3}(k+1)$. By (R5) $D$ can get at least $\frac{1}{3}(k+1)$ from $f$.
\end{proof}

Since $D$ has no chord, every non-internal triangle contains two edges of $E(D,V(G) - D)$. Moreover, by Lemma ~\ref{lemma:D1}, every $T_i$-subgraph meets $D$ by at most one edge for $i \in \{1,2\}$.  Let $t_i$ be the number of $T_i$-subgraphs formed by non-internal $3$-faces for $i \in \{1,2\}$ and let $s'$ be the number of edges in $E(D, V(G) - D)$ that are not in any $3$-faces. Then $f_3= t_1 + 2t_2$ and $s = s' + 2 t_1 + 3 t_2$. Now by (R4) and (R5),
\begin{align}
\mu^*(D)&=d(D)+4+\sum_{v\in D} (d(v)-4)-f_3+b\\
&=d(D)+4+\sum_{v\in D} (d(v)-2)-2d(D)-f_3+b\\
&=4-d(D)+s-f_3+b\\
&\ge4-d(D)+s'+f_3'+b
\end{align}
where $f_3' = t_1+t_2$, and $b$ is the charge $D$ receives by (R5).

\begin{lemma}
Let $k=s-f_3=s'+t_1+t_2$,  then $k\geq 1$, and for $d(D)\geq 5$, we have
$$b\geq\left\{
\begin{array}{ll}
\frac{d(D)}{3},& k=1,\\[2mm]
\frac{d(D)-k}{3}, & k\geq 2.
\end{array}
\right.
$$
\end{lemma}

\begin{proof}
Since $V(G)\neq V(D)$, we have $E(D,V(G)-D)\neq \emptyset$, hence $k\geq 1$.

{\bf Case 1:} $k=1$. That is, $s'+t_1+t_2=1$. Let $s'=1$. Then $t_1=t_2=0$. Thus $D$ is adjacent to a $(d(D)+2)^{+}$-face that shares at least $d(D)$ edges with $D$.
Since $d(D)+2\geq 7$, by Lemma~\ref{lemma:D1} part (d), the $(d(D)+2)^{+}$-face can send at least $\frac{d(D)+1}{3}$
to $D$. Therefore, $b\geq \frac{d(D)+1}{3}>\frac{d(D)}{3}$. Now let $s'=0$. Then $t_1=1$ or $t_2=1$. Thus $D$ is adjacent to a $(d(D)+1)^{+}$-face that shares at least $d(D)-1$ edges with $D$.
Since $d(D)\geq 5$, $d(D)+1\geq 6$, by Lemma~\ref{lemma:D1} part(d), the $(d(D)+1)^{+}$-face can send at least $\frac{(d(D)-1)+1}{3}$ to $D$. Therefore, $b\geq \frac{d(D)}{3}$.

{\bf Case 2:} $k\geq 2$. Suppose that the non-internal $3$-faces meets $D$ by $p$ edges, the other edges of $D$ are divided into $q$ segments by the $s'$ edges and the $T_i$-subgraphs, and each segment has $d_i$ edges. Then we have
$p\leq t_1+t_2=k-s'$, and $d_1+d_2+\cdots +d_q= d(D)-p\geq d(D)-k+s'$.
If $d_i\geq 2$, then by Lemma~\ref{lemma:D1} part (a), these $d_i$ edges are in a $5^{+}$-face, by Lemma~\ref{lemma:D1} part (c) and (d), this $5^{+}$-face sends at least $\frac{d_i}{3}$ to $D$. If $d_i=1$, let the edge be $e_i=u_iv_i$. If $e_i$ is in a $5^{+}$-face, then by Lemma~\ref{lemma:D1} (c) and (d),
 this $5^{+}$-face sends at least $\frac{d_i}{3}$ to $D$. If $e_i$ is in a $4$-face,
 then $u_i$ and $v_i$ are both incident to some edges in $E(D,V(G)-D)$ that are not in any $3$-faces.

Suppose that there are $t$ such edges, each edge is in a $4$-face, then $t\leq s'$. Without loss of generality, let $d_1=d_2=\cdots=d_t=1$,  and $e_i$ is in a $4$-face,
 then we have $d_{t+1}+d_{t+2}+\cdots +d_q\geq d(D)-k+s'-t\geq d(D)-k$. For $t+1\leq i\leq q$, let $f_i$ be a $5^+$-face that shares $d_i$ edges with $D$, then each $f_i$ gives at least $\frac{d_i}{3}$ to $D$, hence,
$b\geq \frac{d_{t+1}}{3}+\frac{d_{t+2}}{3}+\cdots +\frac{d_q}{3}\geq \frac{d(D)-k}{3}$.
\end{proof}

Now, if $d(D)\in \{3,4\}$, since $k\geq 1$, $\mu^*(D)=4-d(D)+s-f_3+b>0$. So we assume that $d(D)\in \{5,6,7\}$. If $s-f_3=k=1$, we have
$\mu^*(D)=4-d(D)+s-f_3+b\ge4-d(D)+1+ \frac{d(D)}{3}
=5- \frac{2d(D)}{3}\ge 5- \frac{2}{3}\cdot 7>0$
If $s-f_3=k\geq 2$, we have
$\mu^*(D)=4-d(D)+s-f_3+b\ge4-d(D)+k+\frac{d(D)-k}{3}
\ge 4- \frac{2}{3}\cdot 7+ \frac{4}{3}>0$.

This completes the proof of Theorem~\ref{strong1}.

\begin{proof}[{\bf Proof of Theorem~\ref{main1} using Theorem~\ref{strong1}: }]
We may assume that $G$ contains a $3$-cycle $C$, for otherwise, $G$ is DP-$4$-colorable by Theorem~\ref{thm:no-3456-DP-coloring}. By Theorem~\ref{strong1}, any precoloring of $C$ can be extended to $G$, so $G$ is also DP-$4$-colorable.
\end{proof}

\section{Proof of Theorem~\ref{main2}}\label{proof2}

We prove the following theorem, which is stronger than Theorem~\ref{main2}. Recall that we use $\mathcal{G}_2$ to denote the class of planar graphs without $4$-cycles adjacent to $6$-cycles.

\begin{theorem}\label{strong2}
If $G\in\mathcal{G}_2$, then any pre-coloring of a good $8^-$-cycle can be extended to a DP-$4$-coloring of $G$.
\end{theorem}

Let $(G,C_0)$ be a minimal counterexample to Theorem~\ref{strong1}.  That is, $G \in \mathcal{G}_2$, $C_0$ is a good $8^-$-cycle in $G$ that is pre-colored with a $DP$-$4$-coloring $\phi_0$ that can not be extended to a $DP$-$4$-coloring of $G$, and $|V(G)|$ is as small as possible. Consider a planar embedding of $G$. By Lemma~\ref{c0-nonseparating}, we may assume that $C_0$ is the boundary of the outer face $D$ of $G$.

\begin{lemma}\label{chordless}
$D$ has no chord, and moreover, if $v_1,v_2\in V(D)$ and $v_1v_2\notin E(G)$, then $v_1$ and $v_2$ have no common neighbors in int(D).
\end{lemma}
\begin{proof}
Suppose otherwise that $D$ has a chord. Then $|D|\ge4$. If $|D|=4$, then $V(G)=V(D)$ since $G$ contains no separating $3$-cycles by Lemma~\ref{separating}. If $|D|=5$, then the $4$-cycle created by the chord can not be a bad $4$-cycle since $D$ itself is not a bad cycle.
By Lemma~\ref{separating}, $G$ contains no separating good $4$-cycles, and hence, $V(G) = V(D)$.

If $|D|=6$, then since $G \in \mathcal{G}_2$, the chord must form a $3$-cycle and a $5$-cycle together with edges on $D$.  By Lemma~\ref{separating}, $G$ contains no separating $5$-cycles, and therefore, $V(G)=V(D)$.

If $|D|=7$, then the chord will form either a $3$-cycle and a $6$-cycle or a $4$-cycle and a $5$-cycle with edges of $D$. Note that, if a $4$-cycle is created by the chord, then it can not be a bad $4$-cycle since $D$ itself is a good cycle.  We again get that $V(G)=V(D)$ since $G$ contains no separating good $6^-$-cycles by Lemma~\ref{separating}.

If $|D|=8$, then since $G \in \mathcal{G}_2$, the chord will form either a $3$-cycle and a $7$-cycle or two $5$-cycles with edges of $D$.  So $V(G)=V(D)$ since $G$ contains no separating good  $7^-$-cycles by Lemma~\ref{separating}.

Now assume that $v_1$ and $v_2$ are two non-adjacent vertices on $D$ that have a common neighbor $u \in int(D)$. Then by Lemma~\ref{minimum} $d(u)\ge 4$. Let $P$ be the shorter path between $v_1$ and $v_2$ on $D$. Note that $v(P)\ge3$. If $v(P)\ge4$, then $P+u$ or $D-P+\{v_1,v_2,u\}$ is a separating $7^-$-cycle since $G$ contains no $4$-cycles adjacent to $6$-cycles, contrary to Lemma~\ref{separating}. So we may assume that $v(P)=3$ and $P=v_1wv_3$. By Lemma~\ref{separating}, $v_1uv_2w$ is either a bad $4$-cycle or a $4$-face. In both cases, $D-w+u$ cannot be a bad $8$-cycle since $G$ contains no $4$-cycles adjacent to $6$-cycles. So $D-w+u$ is a separating good $8$-cycle, contrary to Lemma~\ref{separating}.
\end{proof}

By Lemma~\ref{good-cycle} and ~\ref{separating}, $G$ has no separating $3$-cycle. So every $6$-face in $G$ is bounded by a $6$-cycle. We require the following structural results on $5^{-}$-faces of $G$.

\begin{lemma}\label{4444}
Two internal $(4,4,4)$-faces cannot share exactly one common edge in $G$.
\end{lemma}

\begin{proof}
Suppose to the contrary that $T_1=uvx$ and $T_2=uvy$ share a common edge $uv$. Let $S=\{u,v,x,y\}$ and $G'=G-S$. For each $v\in V(G')$, let $L'(v)=L(v)$ and let $H'=H\setminus\{L(w): w\in S\}$. By the minimality of $G$, the graph $G'$ has an $DP$-$4$-coloring. Thus there is an independent set $I'$ in $H$ with $|I'|=|V(G)|-4$.

For each $w\in S$, we define that $$L^*(w)=L(w)\setminus \{(w,k): (w,k)(u,k)\in E(H), u\in N(w) \text{ and } (u,k)\in I'\}.$$ Since $|L(v)|\ge4$ for all $v\in V(G)$, we have $$|L^*(u)|\ge3,\  |L^*(v)|\ge3,\  |L^*(x)|\ge2,\  |L^*(y)|\ge2.$$ So we can choose a vertex $(v,c)$ in $L^*(v)$ for $v$ such that $L^*(x)\setminus \{(v,c): (v,c)(x,c)\in E(H)\}$ has at least two available colors. Color $y,u,x$ in order, we can find an independent set $I^*$ with $|I^*|=4$. So $I'\cup I^*$ is an independent set of $H$ with $|I'\cup I^*|=|V(G)|$, a contradiction.
\end{proof}

\begin{lemma}\label{corollary}
Let $f=v_1v_2v_3v_4v_5$ be an internal $5$-face that is adjacent to an internal $(4,4,4)$-face $v_1v_2v_{12}$. If $v_3$ is a $4$-vertex, then $v_2$ has a neighbor on $D$. Consequently, (i) neither $v_1v_5$ nor $v_2v_3$ is on an internal $(4,4,4)$-face; (ii) If $v_3v_4$ is on an internal $(4,4,4)$-face, then each of $v_2$ and $v_3$ has a neighbor on $D$.
\end{lemma}
\begin{proof}
Suppose that $v_3$ is a $4$-vertex and $v_2$ has no neighbor on $D$. Let $N(v_2)=\{v_1,v_{12},u,v_3\}$. First we show that there is no bad $4$-cycle contains $v_{1}$, $v_2$, and $u$. Suppose otherwise that $C$ is a bad $4$-cycle with $\{v_1, v_2, u\} \subset V(C)$. Since $G \in \mathcal{G}_2$, the vertex $v_1$ is not adjacent to $u$. Therefore, the cycle $C$ uses the edges $v_1v_2$ and $v_2u$. Note that $v_{12}$ and $v_3$ are the only neighbors of $v_2$ not on $C$. Since $C$ is a bad cycle, either $v_{12}$ or $u$ is adjacent to every vertex on $C$. In both cases, $G$ would contain a $4$-cycle adjancent to a $6$-cycle; a contradiction.

Next we show that the graph $G'$ obtained by identifying $v_1$ and $u$ of $G-\{v_{12},v_2,v_3\}$ has no $4$-cycles adjacent to $6$-cycles. Suppose otherwise.  Then either a new $4$-cycle or a new $6$-cycle is created by the identification. So there exists a path $P$ of length $4$ or $6$ between $v_1$ and $u$ in $G-\{v_{12}, v_2, v_3\}$. Let $C$ be the cycle in $G$ formed by the path $P$ and the edges $v_1v_2$ and $v_2u$.

If $P$ has length $4$, then $C$ is a separating $6$-cycle of $G$, contrary to Lemma~\ref{separating}.


If $P$ has length $6$, then $C$ is a separating $8$-cycle. By Lemma~\ref{separating}, $C$ is a bad cycle. Since $G \in \mathcal{G}_2$, $G$ can not have a subgraph isomorphic to the fourth configuration in Figure ~\ref{badcycle}; moreover, if $G$ has a subgraph isomorphic to the third configuration in Figure ~\ref{badcycle}, then neither $v_1$ nor $v_2$ can be a vertex in one of the three triangles in Figure ~\ref{badcycle}. It follows that $G$ contains a separating $7$-cycle, a contradiction. Therefore, we may assume that $C$ is a boundary cycle of a subgraph isomorphic to the second configuration of Figure ~\ref{badcycle}. Note that every edge of $C$ is in a $6$-cycle of $G$, and hence, is not adjacent to a $4$-cycle. It follows that the new $6$-cycle created by the identification is not  adjacent to any $4$-cycle. So we conclude that $G'$ has no $4$-cycles adjacent to $6$-cycles.

By Lemma~\ref{reduce} (b), at most one of $v_{12}$ and $v_3$ is a $4$-vertex, a contradiction. Therefore, $u\in D$.
\end{proof}

We are now ready to complete the proof of Theorem~\ref{strong2} by a discharging procedure.  For each $x\in V\cup F\backslash D$, let $x$ have an initial charge of $\mu(x)=d(x)-4$, and $\mu(D)=d(D)+4$. By Euler's Formula, $\sum_{x\in V\cup F}\mu(x)=0$.

Let $\mu^*(x)$ be the charge of $x\in V\cup F$ after the discharging procedure. To lead to a contradiction, we shall prove that $\mu^*(x)\ge 0$ for all $x\in V\cup F$ and $\mu^*(D)>0$.

We call a $5$-vertex $v$ {\em{bad}} if $v$ is on three $3$-faces and exactly two of which are adjacent; otherwise, it is {\em{good}}. We call a $4$-vertex {\em{bad}} if it is on exactly two adjacent $3$-faces. It is easy to check that every $7^+$-face shares at most two bad $4$-vertices with an adjacent $3$-face.  We call a $4^-$-face $f_0$ {\em{special}} to a $5$-face $f$ if either $f_0$ is a $3$-face sharing two internal vertices with $f$ and exactly one vertex with $D$ or a $4$-face sharing two internal vertices with $f$ and two vertices with $D$ and $f_0$ is adjacent to no $3$-faces other than $D$.

The discharging rules:
All $3^+$-faces mentioned here are distinct from $D$.
\begin{enumerate}[(R1)]

\item Every internal $6^+$-vertex gives $\frac{1}{2}$ to each incident $3$-face, every good internal $5$-vertex distributes its charge evenly to incident $3$-faces, and every bad internal $5$-vertex gives $\frac{1}{4}$ to its incident isolated $3$-face and $\frac{3}{8}$ each to the other two incident $3$-faces.

\item Each $5$-face gives $\frac{1}{3}$ to every adjacent internal $(4,4,4)$-face, $\frac{1}{6}$ to every other adjacent non-special $3$-face or non-special $4$-face.

\item Each $4$- or $6$-face gives $\frac{1}{3}$ to each adjacent $3$-face.

\item Each $7^+$-face $f$ gives $\frac{6}{7}t$ to each adjacent $3$-face sharing two bad $4$-vertices with $f$,  $\frac{9}{14}t$ to each adjacent $3$-face sharing exactly one bad $4$-vertex with $f$ and $\frac{3}{7}t$ to each adjacent $4$-face or $3$-face not sharing a bad $4$-vertex with $f$, where $t$ is the number of common edges of $f$ and the $4^-$-face;

\item The outer-face $D$ gets $\mu(v)$ from each incident vertex and gives $\frac{5}{7}$ to each non-internal $3$-face.

\item After the above rules, each $5^+$-face gives its surplus charge to $D$.
\end{enumerate}

\textbf{Remark:} By (R4), we may consider that the charge $f$ gives to its adjacent $3$-faces is carried by edges of $f$. Then on average, every edge of $f$ carries a charge of at most $\frac{3}{7}$. We may consider it in the following way: Suppose that $f=v_1v_2v_3\ldots v_k$. Suppose that $v_1v_2$ is on a $3$-face such that $v_2$ is a bad $4$-vertex, then since $G \in \mathcal{G}_2$, every face that contains the edge $v_2v_3$ is a $7^+$-face. So the edge $v_2v_3$ carries a charge of $0$, we then let the edge $v_1v_2$ give a charge of $\frac{3}{14}$ to the edge $v_2v_3$. Repeat the same procedure for every bad $4$-vertex of $f$ that is in a $3$-face adjacent to $f$.

\medskip
We first check the final charge of vertices in $G$. Let $v$ be a vertex in $G$. If $v\in D$, then by (R6),  $\mu^*(v)\ge0$. If $v\not\in D$, then by Lemma~\ref{minimum} $d(v)\ge4$. If $d(v)=4$, then $v$ is not involved in the discharging procedure. So $\mu^*(v)=\mu(v)=d(v)-4=0$. If $d(v)\ge5$, then by (R1), when $v$ is a good $5$-vertex, $\mu^*(v)\ge 5-4-\frac{1}{2}\cdot2=0$;  when $v$ is a bad $5$-vertex, $\mu^*(v)\ge5-4-\frac{1}{4}-\frac{3}{8}\cdot2=0$; when $v$ is a $6^+$-vertex, $\mu^*(v)\ge d(v)-4-\frac{1}{2}(d(v)-2)=0$.

Now we check the final charge of faces other than $D$ in $G$. Let $f$ be a face in $G$. By the rules, we only need to show that after (R1)-(R5), $f$ has nonnegative charge, and we denote this charge by  $\mu'(f)$.
First assume that $d(f)=4$.  Then by Lemma~\ref{chordless}, $|V(f)\cap V(D)|\le2$. Then $f$ is adjacent to at least two $5^+$-faces other than $D$, and hence by (R2) and (R3), $\mu'(f)\ge4-4 +\frac{1}{6}\cdot 2-\frac{1}{3}=0$.

Next we assume that $d(f)=5$. Then by Lemma~\ref{corollary}, $f$ is adjacent to at most two internal $(4,4,4)$-faces. If $f$ is adjacent to at most one $(4,4,4)$-face, then by (R2),  $\mu'(f)\ge5-4-\frac{1}{3}-\frac{1}{6}\cdot 4=0$. So by Lemma~\ref{corollary}, we may assume that $f=v_1v_2v_3v_4v_5$ and each of $v_1v_2$ and $v_3v_4$ is on a $(4,4,4)$-face. Therefore, each of $v_2$ and $v_3$ has a neighbor on $D$. It is easy to check that if $v_2v_3$ is on a $4^-$-face $f_0$, then since $G \in \mathcal{G}_2$, $f_0$ is not adjacent to any $3$-faces. Therefore, the face $f_0$ is special to the $5$-face $f$; by (R2), $f_0$ gets no charge from $f$.  So $\mu'(f)\ge5-4-\frac{1}{3}\cdot2-\frac{1}{6}\cdot2=0$.

For a $6$-face $f$, it follows from (R4) that $\mu'(f)\ge6-4-\frac{1}{3}\cdot6=0$.

Assume that $d(f)\ge7$. By the Remark, $\mu'(f)\ge d(f)-4-\frac{3}{7}d(f)\ge0$.


Suppose that $d(f)=3$. By Lemma~\ref{property2}(a), $f$ must be in $T_i$-subgraph for $i \in \{1, 2, 3, 4\}$.  Define $\mu(T_i):=\mu(f_1)+\mu(f_2)+\ldots+\mu(f_i)=-i$
where $f_1$, $f_2\ldots f_i$ are $3$-faces in $T_i$
and define $\mu'(T_i):=\mu'(f_1)+\mu'(f_2)+\ldots+\mu'(f_i).$
Since we can redistribute all the charge received by $T_i$ evenly to each $3$-face contained in it, it suffices to show that $\mu'(T_i)\geq 0$. Let $f=uvw$ and let $f_1, f_2$ and $f_3$ be the faces sharing edges $uv,vw,uw$ with $f$ respectively.

\textbf{Case 1:} $f$ is in $T_1$.

By Lemma~\ref{chordless}, $|V(f)\cap V(D)|\le2$. If $|V(f)\cap V(D)|\ge1$, then $f$ gets $\frac{5}{7}$ from $D$, at least $\frac{1}{6}$ from each adjacent $4^+$-face other than $D$ by (R2)-(R5). So $\mu'(H_1)\ge-1+\frac{5}{7}+\frac{1}{6}\cdot2>0$. So we may assume that $f$ is an internal $3$-face, and hence, each of $u$, $v$, and $w$ is a $4^+$-vertex.
If $f$ is not adjacent to any $5$-faces, then each face adjacent to $f$ is either a $6^+$-face or a $4$-face that has at most two common vertices. So $f$ gets at least $\frac{1}{3}\cdot3$ from adjacent faces by (R3)(R4). So $\mu'(f)\ge0$. So we may assume that $f$ is adjacent to a $5$-face, by symmetry, say $f_1$ is a $5$-face. If $f$ is a $(4,4,4)$-face, then by (R2), (R3) and (R4),  $f$ gets $\frac{1}{3}\cdot3$ from adjacent faces. So $\mu'(f)\ge-1+1=0$. If $f$ is a $(4,4,5^+)$-face, then $f$ gets $\frac{1}{6}$ from $f_1$ by (R2). If $d(w)\ge6$ or $w$ is a good $5$-vertex, then $w$ gives $\frac{1}{2}$ to $f$ by (R1). Note that $f$ gets at least $\frac{1}{6}\cdot2$ from $f_2$ and $f_3$. So $\mu'(f)\ge-1+\frac{1}{6}\cdot3+\frac{1}{2}=0$. If $w$ is a bad $5$-vertex, then both $f_2$ and $f_3$ are $7^+$-faces. So by (R2) and (R4), $f$ gets at least $\frac{3}{7}\cdot2$ from each of $f_2$ and $f_3$;  and gets $\frac{1}{4}$ from $w$ by. So $\mu'(f)\ge-1+\frac{3}{7}\cdot2+\frac{1}{4}>0$. If $f$ is a $(4,5^+,4)$-face, then $v$ is either a $6^+$-vertex or a good $5$-vertex. So $f$ gets $\frac{1}{2}$ from $v$ by (R1). Note that $f$ gets at least $\frac{1}{6}\cdot3$ from adjacent faces. So $\mu'(f)\ge-1+\frac{1}{2}+\frac{1}{2}=0$. If $f$ is a $3$-face with at least two $5^+$-vertices, then $f$ gets at least $\frac{1}{4}\cdot2$ from the two $5^+$-vertices and $\frac{1}{6}\cdot3$ from adjacent faces. So $\mu'(f)\ge-1+\frac{1}{2}+\frac{1}{2}=0$.

Note that for $i\in\{2,3,4\}$, by Lemma~\ref{property2}(b), each boundary edge of $T_i$ is adjacent to a $7^+$-face.

\textbf{Case 2:} $f$ is in $T_2$.

Let $V(T_2)=\{u,v,x,y\}$ such that each of $v$ and $x$ is on two $3$-faces.
By Lemma~\ref{chordless}, $|V(T_2)\cap V(D)|\le3$. If $|V(T_2)\cap V(D)|=1$, then $T_2$ gets at least $\frac{5}{7}$ from $D$ and at least $\frac{3}{7}\cdot4$ from adjacent $7^+$-faces other than $D$. So $\mu'(T_2)\ge-2+\frac{5}{7}+\frac{12}{7}>0$.  If $|V(T_2)\cap V(D)|\in\{2,3\}$, then $T_2$ gets $\frac{5}{7}\cdot2$ from $D$ and at least $\frac{3}{7}\cdot2$ from adjacent $7^+$-faces other than $D$. So $\mu'(T_2)\ge-2+\frac{10}{7}+\frac{6}{7}>0$.  So we may assume that $|V(T_2)\cap V(D)|=0$.
If both $v$ and $x$ are $4$-vertices, then by (R4),  $T_2$ gets $\frac{9}{14}\cdot4$
from adjacent $7^+$-faces. If at least one of $v$ and $x$ is a $5^+$-vertex, then $T_2$ gets at least $\frac{3}{8}\cdot2$ from a $5^+$-vertex and $\frac{3}{7}\cdot4$ from adjacent $7^+$-faces. In either case, $\mu'(T_2)\ge-2+\min\{\frac{18}{7}, \frac{69}{28}\}>0$.

\textbf{Case 3:} $f$ is in $T_3$.

Let $u,v,w,x,y$ be the five vertices of $T_3$ in cyclic order such that $y$ is on three $3$-faces of $T_3$. Since $G$ has no separating $3$-cycles, the subgraph $T_3$ contains a pair of non-adjacent $3$-faces. 

First assume that $|V(T_3)\cap V(D)|\ge1$. If both $v$ and $w$ are on $D$, then by Lemma~\ref{chordless} $y$ is not on $D$. So each of $uy$ and $yx$ is on a $7^+$-face other than $D$. By (R4) and (R5),  $T_3$ gets $\frac{5}{7}\cdot3$ from $D$ and at least $\frac{3}{7}\cdot2$ from adjacent $7^+$-faces other than $D$. So $\mu'(T_3)\ge-3+\frac{15}{7}+\frac{6}{7}=0$. If exactly one of $v$ and $w$ is on $D$, say $v$, then $T_3$ gets at least $\frac{5}{7}\cdot2$ from $D$. By Lemma~\ref{chordless}, at most one of $u$ and $y$ is on $D$. So $T_3$ gets at least $\frac{3}{7}\cdot4$ from adjacent $7^+$-faces other than $D$. Therefore, $\mu'(T_3)\ge-3+\frac{10}{7}+\frac{12}{7}>0$. So we may assume that neither $v$ nor $w$ is on $D$. If $y$ is on $D$, then $T_3$ gets at least $\frac{5}{7}\cdot3$ from $D$ and $\frac{3}{7}\cdot3$ from adjacent $7^+$-faces other than $D$. Therefore,  $\mu'(T_3)\ge-3+\frac{15}{7}+\frac{9}{7}>0$. So we may assume that none of $v$, $w$, and $y$ is on $D$. By our assumption, one of $u$ and $x$ is on $D$. If both $u$ and $x$ are on $D$, then $T_3$ gets $\frac{5}{7}\cdot2$ from $D$ and $\frac{3}{7}\cdot5$ from adjacent $7^+$-faces other than $D$. So $\mu'(T_3)\ge-3+\frac{10}{7}+\frac{15}{7}>0$. If exactly one of $u$ and $x$ is on $D$, say $u$, then by Lemma~\ref{4444}, at least one vertex in $\{v,w,x,y\}$ is a $5^+$-vertex. So $T_3$ gets at least $\frac{1}{4}$ from the $5^+$-vertex by (R1). Note that by (R4) and (R5), $T_3$ gets $\frac{3}{7}\cdot5$ from adjacent $7^+$-faces other than $D$ and $\frac{5}{7}$ from $D$. So $\mu'(T_3)\ge-3+\frac{5}{7}+\frac{1}{4}+\frac{15}{7}>0$.

Next we assume that $V(H_3)\cap V(D)=\emptyset$.
If both $v$ and $w$ are $4$-vertices, then by (R4), $\mu'(T_3)\ge-3+\frac{9}{14}\cdot2+\frac{6}{7}+\frac{3}{7}\cdot2=0$. If exactly one of $v$ and $w$ is a $4$-vertex, then by (R1) and (R4), $\mu'(T_3)\ge-3+\frac{9}{14}\cdot2+\frac{3}{7}\cdot3+\frac{3}{8}\cdot 2>0$. If neither $v$ nor $w$ is a $4$-vertex, then by (R1) and (R4) again, $\mu'(T_3)\ge-3+\frac{3}{7}\cdot5+\frac{3}{8}\cdot 4>0$.

\textbf{Case 4:} $f$ is in $T_4$.

By Lemma~\ref{property2}(a), $T_4$ is isomorphic to the wheel graph $W_4$ with $4$ spokes.  Let $v$ be the center $4$-vertex and $N(v)=\{v_1,v_2,v_3,v_4\}$. Clearly $v \notin V(D)$ since $D$ is the outer face.

Suppose first that $|V(T_4)\cap V(D)|\ge1$. By Lemma~\ref{chordless}, $|V(T_4)\cap V(D)|\le3$. If $|V(T_4)\cap V(D)|=3$, say $v_4$ is not on $D$, then by (R5), $T_4$ gets $\frac{5}{7}\cdot4$ from $D$. If $d(v_4)=4$, then by (R4), $T_4$ gets at least $\frac{9}{14}\cdot2$ from adjacent $7^+$-faces other than $D$. Therefore, $\mu'(T_4)\ge-4+\frac{20}{7}+\frac{9}{7}>0$. If $d(v_4)\ge5$, then by (R1) and (R4), $T_4$ gets at least  $\frac{3}{7}\cdot2+\frac{3}{4}$ from adjacent $7^+$-faces and $v_4$. Therefore, $\mu'(T_4)\ge-4+\frac{20}{7}+ \frac{3}{7}\cdot2+\frac{3}{4}>0$. If $|V(T_4)\cap V(D)|=2$, then by symmetry either $v_1$ and $v_2$ are on $D$ or $v_1$ and $v_3$ are on $D$. In the former case, $T_4$ gets $\frac{5}{7}\cdot3$ from $D$. If both $v_3$ and $v_4$ are $4$-vertices, then by (R4), $T_4$ gets at least $\frac{9}{14}\cdot2+\frac{6}{7}$ from adjacent $7^+$-faces. So $\mu'(T_4)\ge-4+\frac{15}{7}+\frac{15}{7}>0$. If at least one of $v_3$ and $v_4$ is $5^+$-vertex, then by (R1) and (R4), $T_4$ gets at least $\frac{3}{7}\cdot3+\frac{3}{4}$ from a $5^+$-vertex and adjacent $7^+$-faces. So $\mu'(T_4)\ge-4+\frac{15}{7}+\frac{9}{7}+\frac{3}{4}>0$. In the latter case, $T_4$ gets $\frac{5}{7}\cdot4$ from $D$ and at least $\frac{3}{7}\cdot4$ from adjacent $7^+$-faces. So $\mu'(T_4)\ge -4+\frac{20}{7}+\frac{12}{7}>0$. If $|V(T_4)\cap V(D)|=1$, say $v_1$ is on $D$, then by Lemma~\ref{4444}, at least one of $v_2$, $v_3$, and $v_4$ is a $5^+$-vertex. If at least two of $v_2$, $v_3$, and $v_4$ are $5^+$-vertices, then by (R1) and (R4), $T_4$ gets at least $\frac{3}{7}\cdot4+\frac{3}{4}\cdot2$ from incident $5^+$-vertices and adjacent $7^+$-faces. Therefore, $\mu'(T_4)\ge -4+\frac{10}{7}+\frac{12}{7} + \frac{3}{2} >0$. If exactly one of $v_2$, $v_3$, and $v_4$ is $5^+$-vertex, then by symmetry, either $v_2$ or $v_3$ is a $5^+$-vertex. So by (R1) and (R4), $T_4$ gets at least $\min\{\frac{3}{7}+\frac{6}{7}+\frac{9}{14}\cdot2, \frac{9}{14}\cdot4\} = \frac{18}{7}$ from adjacent $7^+$-faces and at least $\frac{3}{4}$ from incident $5$-vertices. Note that $T_4$ gets $\frac{5}{7}\cdot2$ from $D$ by (R5). Therefore, $\mu'(T_4)\ge -4+\frac{10 }{7}+ \frac{3}{4} + \frac{18}{7}>0$.

So we may assume that $V(T_4)\cap V(D)=\emptyset$.
By Lemma~\ref{4444}, at most two vertices in $N(v)$ are $4$-vertices. If none of the vertices in $N(v)$ is $4$-vertex, then $\mu'(T_4)\ge-4+\frac{3}{7}\cdot4+\frac{3}{4}\cdot4>0$. If exactly one vertex of $N(v)$ is $4$-vertex, then $\mu'(T_4)\ge-4+\frac{9}{14}\cdot2+\frac{3}{7}\cdot2+\frac{3}{4}\cdot3>0$. If two vertices of $N(v)$ are $4$-vertices, then by Lemma~\ref{reduce} the two $4$-vertices must be adjacent. So $\mu'(T_4)\ge-4+\frac{9}{14}\cdot2+\frac{6}{7}+\frac{3}{7}+\frac{3}{4}\cdot2>0$.

\medskip

Finally, we check the final charge of $D$ and show that $\mu^*(D)>0$. Let $f_3$ be the number of $3$-faces sharing vertices with $D$ and let $b$ be the charge that $D$ gets from other faces by (R7). Let $E(D, V(G)-D)$ be the set of edges between $D$ and $V(G)-D$ and let $s$ be its size.  Let $s'$ be the number of edges in $E(D, V(G)-D)$ that are not on any $3$-faces and $f_3'$ be the number of $T_i$-subgraphs in which every 3-face intersects with $D$. 

\begin{lemma}\label{46D}
The following statements are true about $D$.

(a) Let $T_i$ be a subgraph of $G$ such that each $3$-face in $T_i$ intersects $D$. Then $i\in[3]$ and $T_i$ shares at most two edges with $D$.

(b) Let $f$ be a $5^+$-face that shares $k$ edges with $D$. If $d(f)=5$, then $f$ gives at least $\frac{1}{6}k$ to $D$; if $d(f)=6$, then $f$ gives at least $\frac{1}{3}k$ to $D$; if $d(f)\ge7$, then $f$ gives at least $\frac{3}{7}k$ to $D$.

(c) $b\ge\frac{1}{3}(d(D)-3f_3'-s')$.
\end{lemma}
\begin{proof}
(a) If $i=4$, or $i\in[3]$ and $T_i$ shares at least three edges with $D$, then $D$ either has a chord or two non-adjacent vertices that have common neighbors in $int(D)$, contrary to Lemma~\ref{chordless} in either case.


(b) Let $f$ be a $5$-face. Then $f$ is adjacent to at most one internal $(4,4,4)$-face.  So by (R2), it requires to give out at most $\frac{1}{6}$ to $4^-$-faces other than the $(4,4,4)$-face, so it gives at least $\frac{1}{6}k$ to $D$ by (R6). By (R3) and (R6), when $f$ is a $6$-face it gives at least $\frac{1}{3}k$ to $D$. When $f$ is a $7^+$-face, by (R4), (R6) and the Remark, it gives at least $\frac{3}{7}k$ to $D$.  

(c) By (a), $D$ contains at least $(d(D)-3f_3'-s')$ vertices of degree two. By Lemma~\ref{chordless}, each $2$-vertex on $D$ is on a $5^+$-face. when a $2$-vertex is isolated, that is, it is not adjacent to another $2$-vertex, it must be on a $5^+$-face and by (b), $D$ gets at least $\frac{1}{3}$ from the $5^+$-face through the two edges incident to the $2$-vertex;  when a $2$-vertex is not isolated, it must be on a $6^+$-face,  so by (b), $D$ gets at least $\frac{1}{3}$ through each edge incident to the $2$-vertices. So $D$ gets at least $\frac{1}{3}(d(D)-3f_3'-s')$ from adjacent $5^+$-faces, that is $b\ge\frac{1}{3}(d(D)-3f_3'-s')$.
\end{proof}

Note that $s=s'+f_3+f_3'$. Then by (R6) and (R7) and Lemma~\ref{46D}(c),
\begin{align}
\mu^*(D)&=d(D)+4+\sum_{v\in D} (d(v)-4)-\frac{5}{7}f_3+b\\
&=d(D)+4+\sum_{v\in D} (d(v)-2)-2d(D)-\frac{5}{7}f_3+b\\
&=4-d(D)+s-\frac{5}{7}f_3+b\\
&\ge4-d(D)+s'+f_3'+f_3-\frac{5}{7}f_3+\frac{1}{3}(d(D)-3f_3'-s')\\
&\ge4-\frac{2}{3}d(D)+\frac{2}{3}s'+\frac{2}{7}f_3
\end{align}
Note that $s'+f_3\ge1$ since $G\ne D$. So $\mu^*(D)\le0$ only if $d(D)\in\{7,8\}$.
\begin{itemize}
\item $d(D)=7$. Then $\mu^*(D)\le0$ only if $\frac{2}{3}s'+\frac{2}{7}f_3\le\frac{2}{3}$. So $f_3=0,s'=1$ or $f_3\in[2],s'=0$. If $f_3=0$, then $D$ shares at least $7$ edges with a $9^+$-face. By Lemma~\ref{46D}(b), $b\ge\frac{3}{7}\cdot7=3$. So $\mu^*(D)\ge-3+1+3>0$. If $f_3=1$, then $D$ shares at least $6$ edges with a $7^+$-face. By Lemma~\ref{46D}(b), $\mu^*(D)\ge-3+1+\frac{2}{7}+\frac{18}{7}>0$. If $f_3=2$, then either $f_3'=1$ and $D$ shares at least $5$ edges with a $7^+$-face or $f_3'=2$ and $D$ shares at least $5$ edges with $5^+$-faces. In either case, by Lemma~\ref{46D}(b), $\mu^*(D)\ge-3+\frac{4}{7}+\min\{1+\frac{15}{7},2+\frac{5}{6}\}>0$.

\item $d(D)=8$. Then $\mu^*(D)\le0$ only if $\frac{2}{3}s'+\frac{2}{7}f_3\le\frac{4}{3}$. So $f_3\le4$ and $s'\le2$. If $s'=2$, then $f_3=0$, thus $D$ shares at least $5$ edges with a $7^+$-face. By Lemma~\ref{46D}(b), $\mu^*(D)\ge-4+2+\frac{15}{7}>0$. So we may assume that $s'\le1$. If $f_3=0$, then $s'=1$ and $D$ shares $8$ edges with a $10^+$-face, so by Lemma~\ref{46D}(b), $\mu^*(D)\ge-4+1+\frac{24}{7}>0$. If $f_3=1$, then $D$ shares at least $7$ edges with a $7^+$-face when $s'=0$, and $D$ either shares at least $5$ edges with a $7^+$-face or shares at least $7$ edges with $6^+$-faces when $s'=1$. In any case, by Lemma~\ref{46D}(b), $\mu^*(D)\ge-4+1+\frac{2}{7}+\min\{\frac{3}{7}\cdot7, 1+\frac{15}{7},1+\frac{7}{3}\}>0$. If $f_3=2$, then $f_3'=1$ or $f_3'=2$. In the former case, $D$ shares at least $6$ edges with $7^+$-faces. So $\mu^*(D)\ge-4+1+\frac{4}{7}+\frac{18}{7}>0$. In the latter case, $D$ either shares at least $4$ edges with a $7^+$-face or shares at least $6$ edges with $6^+$-faces when $s'=0$, and $D$ shares at least $5$ edges with $5^+$-faces when $s'=1$. In any case, $\mu^*(D)\ge-4+2+\frac{4}{7}+\min\{\frac{12}{7},\frac{1}{3}\cdot6,\frac{5}{6}+1\}>0$. If $f_3=3$, then $s'=0$. So $D$ has at least two $2$-vertices when $f_3'=3$, and $D$ shares at least $5$ edges with $7^+$-faces when $f_3'=2$, and $D$ shares at least $6$ edges with $7^+$-faces when $f_3'=1$. In any case, $\mu^*(f)\ge-4+\frac{6}{7}+\min\{3+\frac{2}{3}, 2+\frac{15}{7}, 1+\frac{18}{7}\}>0$. If $f_3=4$, then $s'=0$ and $f_3'=2$. Note that $D$ shares at least $4$ edges with $7^+$-faces. So $\mu^*(D)\ge-4+2+\frac{8}{7}+\frac{12}{7}>0$.
\end{itemize}

This completes the proof for Theorem \ref{strong2}.

\begin{proof}[{\bf Proof of Theorem~\ref{main2} using Theorem~\ref{strong2}: }]
We may assume that $G$ contains a $3$-cycle $C$, for otherwise, $G$ is DP-$4$-colorable by Theorem~\ref{thm:no-3456-DP-coloring}. By Theorem~\ref{strong2}, any pre-coloring of $C$ can be extended to $G$, so $G$ is also DP-$4$-colorable.
\end{proof}

\end{document}